\documentclass{llncs}

\usepackage{llncsdoc}
\usepackage{graphicx}
\usepackage{doc}
\usepackage[dvips,bookmarksopen]{hyperref}

\usepackage{amssymb}
\usepackage{amsmath}
\usepackage{amsfonts}
\usepackage[ruled,lined,linesnumbered,noend]{algorithm2e} 
\usepackage{float}
\usepackage{hyperref}
\usepackage{enumerate}
\usepackage{mathtools}
\usepackage{paralist}

\spnewtheorem{ownclaim}{Claim}{\itshape}{}

\newcommand{\llbracket}{([}
\newcommand{\rrbracket}{])}
\newcommand*{\msl}{\{\mskip-6mu\{}
\newcommand*{\msr}{\}\mskip-6mu\}}

\newcommand{\AC}{\mathsf{AC}}
\newcommand{\NC}{\mathsf{NC}}
\newcommand{\LS}{\mathsf{L}}
\newcommand{\Sym}{\mathsf{Sym}}
\newcommand{\fpf}{\mathsf{fpf}}
\newcommand{\ct}{\mathsf{ct}}
\newcommand{\pe}{\mathsf{pe}}

\DeclareMathOperator{\ord}{ord}
\DeclareMathOperator{\lcm}{lcm}

\allowdisplaybreaks

\begin{document}

\title{Finding cycle types in permutation groups with few generators}
\author{Markus Lohrey \and Andreas Rosowski}
\institute{Universit\"at Siegen, Department ETI \\ \email{$\{$lohrey,rosowski$\}$@eti.uni-siegen.de}}

\maketitle 

\begin{abstract}
The problem whether a given permutation group contains a permutation with a given cycle type is studied.
This problem is known to be \textsf{NP}-complete. In this paper it is shown that the problem can be solved
in logspace for a cyclic permutation group and that it is \textsf{NP}-complete for a 2-generated abelian permutation group.
In addition it is shown that it is \textsf{NP}-complete whether a 2-generated abelian permutation group contains
a fixpoint-free permutation.
\end{abstract}

\section{Introduction}

Permutations are ubiquitous objects in combinatorics \cite{Bona22} and group theory \cite{Cam10}. The set of all permutations on a set $\Omega$ forms
a group $\Sym(\Omega)$ (the \emph{symmetric group} on $\Omega$) under composition. A subgroup of a symmetric group
is called a \emph{permutation group}.  Cayley's famous 
theorem states that every group is isomorphic to a permutation group via the right regular representation.
Here, we only deal with the case that $\Omega$ is finite and write $\Sym(n)$ for $\Sym(\Omega)$ if $|\Omega|=n$.

Having group elements represented as permutations can be often  exploited algorithmically. For instance, the subgroup
membership problem for symmetric groups (Does a given permutation $\pi \in \Sym(n)$ belong to the subgroup generated by given permutations
$\pi_1, \ldots, \pi_k \in \Sym(n)$?) can be solved in polynomial time \cite{FurstHL80,Sims70} and even in $\NC$ \cite{BaLuSe87}. Another problem that has an extremely 
simple algorithm in symmetric groups is the conjugacy problem: given permutations $\pi, \rho \in \Sym(n)$, 
does there exist $\tau \in \Sym(n)$ such that $\pi = \tau^{-1} \rho \tau$? This is equivalent to say that
$\pi$ and $\rho$ have the same  \emph{cycle type}. The cycle type of a permutation $\pi \in \Sym(n)$ specifies for every
$\ell \leq n$ the number of cycles of length $\ell$ when $\pi$ is written (uniquely) as a product of pairwise disjoint cycles.

In this paper we are interested in the problem whether a given permutation group $G \leq \Sym(n)$ (specified by a list of generators) 
contains a permutation of a given cycle type. Or equivalently: does $G$ contain an element that is conjugated to a given
permutation $\pi$? We call this problem {\sf CycleType}.

Cameron and Wu showed in \cite{CameronW10} that {\sf CycleType} is {\sf NP}-complete. Moreover, {\sf NP}-hardness already holds for 
the case where $G$ is an elementary abelian 2-group (i.e., an abelian group where every non-identity element has order two).
Here we further pinpoint the borderline between tractability and non-tractability: We show that if the input permutation group 
$G$ is cyclic and given by a single
generator then {\sf CycleType} can be solved in logarithmic space on a deterministic Turing machine (and hence belongs to the 
complexity class {\sf P}). On the other hand, we show that  {\sf CycleType} is already  {\sf NP}-complete for the case where $G$
is generated by two commuting permutations, i.e., $G = \langle \pi, \tau \rangle$ with $\pi \tau = \tau \pi$.
Moreover, our proof shows that it is already {\sf NP}-complete whether for two given commuting permutations $\pi$ and $\tau$
the coset $\pi \langle \tau \rangle$ (a coset of a cyclic group) contains a permutation with a given cycle type. 

In the last section of the paper, we consider the problem \textsf{FixpointFree} that asks whether a given permutation group contains a fixpoint-free permutation, i.e., a permutation $\pi$ such that $\pi(a) \neq a$ for all $a$. It was shown in \cite{BuchheimJ05,CameronW10} that 
\textsf{FixpointFree}  is {\sf NP}-complete and as for {\sf CycleType}, {\sf NP}-hardness holds already for elementary abelian 2-groups.
The restriction of \textsf{FixpointFree} to cyclic permutation groups is not interesting ($\langle \pi \rangle$ contains a fixpoint-free permutation
if and only if $\pi$ is fixpoint-free). We show that the restriction of \textsf{FixpointFree} to 2-generated abelian permutation groups 
$\langle \pi, \tau \rangle$ is {\sf NP}-complete. Moreover, it is also \textsf{NP}-complete to check whether a coset $\pi \langle \tau \rangle$
of a cyclic permutation group, where in addition $\pi \tau = \tau \pi$, contains a fixpoint-free permutation. 

\paragraph{\bf Related work.} Fixpoint-free permutations are also known as \emph{derangements} and they have received a lot of attention
in combinatorics and group theory; see \cite{cameron2011} for a survey. Jordan proved in 1872 that every permutation group $G$ that acts transitively
on a finite set $\Omega$ of size at least two contains a derangement \cite{Jordan1872}. Arvind  proved that in this situation one can compute in polynomial time a derangement in $G$ \cite{Arvind13}. In the same paper, Arvind shows that the problem whether a given permutation group
$G$ contains a permutation with at least $k$ non-fixpoints is fixed parameter tractable with respect to the parameter $k$.

\section{Preliminaries}
\subsection{General notations}

For integers $1 \leq i \leq j$
we write $[i,j]$ for the set $\{i,i+1,\dots,j\}$  and $[j]$ for $[1,j]$.
For a prime $p$ and an integer $n$ we denote 
with $\nu_p(n)$
the largest positive integer $d$ such that $p^d \mid n$ (it is also called
the $p$-adic valuation of $n$). The greatest common divisor of integers $n_1, \ldots, n_k$ is denoted
by $\gcd(n_1,\ldots, n_k)$ and the least common multiple is denoted by 
$\lcm(n_1,\ldots, n_k)$.

We assume that the reader is familiar with basic concepts of complexity theory; see \cite{AroBar09} for more details.
With $\LS$ (also known as \emph{logspace}) 
we denote the class of all problems that can be solved on a deterministic Turing machine in logarithmic space. 
It is a subset of \textsf{P} (deterministic polynomial time).
\subsection{Permutations} \label{sec-perm}

For $n \geq 1$ we denote with $\Sym(n)$ the group of all permutations on $[n]$.
The identity permutation is denoted by $\mathrm{id}$. 
For $\pi \in \Sym(n)$ and $a \in [n]$ we also write $a\pi$ for $\pi(a)$.
There are two standard representations for a permutation $\pi\in \Sym(n)$:
\begin{compactitem}
\item The {\em pointwise representation} of $\pi$ is the tuple $[\pi(1), \pi(2), \ldots, \pi(n)]$.
\item The {\em cycle representation} is a list $\gamma_1 \gamma_2 \cdots \gamma_k$ of pairwise disjoint cycles. 
Every cycle $\gamma_i$ is written as a list $(a_0, a_1, \ldots, a_{\ell-1})$ (with $a_i \in [n]$) meaning that
 $a_k \pi = a_{k+1 \bmod \ell}$.
Fixpoints (cycles of the form $(i)$) are usually omitted in the cycle representation, but sometimes we will explicitly
list them.
\end{compactitem}

\smallskip
\noindent
Note that every cycle $(a_0, a_1, \ldots, a_{\ell-1})$ can be replaced by a cyclic rotation.
Moreover since disjoint cycles commute, the 
order of the cycles $\gamma_i$ is not relevant.

Computing the pointwise representation from the cycle representation is possible in uniform $\AC^0$ (this is a very small 
circuit complexity class contained in $\LS$).
On the other hand, the cycle representation can be computed in logspace
 from the pointwise representation and no better complexity bound is known \cite{CookM87}.
Therefore, as long as
one works with complexity classes that contain $\LS$  (which will be the case in this paper), there is no reason to specify which of the above two representations of permutations is chosen.

Let $\fpf(n) = \{ \pi \in \Sym(n) \mid a\pi \neq a \text{ for all } a \in [n]\}$ be the set of all \emph{fixpoint-free} permutations.
For $\pi_1, \ldots, \pi_k \in \Sym(n)$ we write $\langle \pi_1,\dots,\pi_k\rangle \leq \Sym(n)$
for the permutation group generated by $\pi_1, \ldots, \pi_k$.
The order $\ord(\pi)$ of $\pi \in \Sym(n)$ is the smallest  integer $i \geq 1$ such that $\pi^i = \mathrm{id}$.  
If $\gamma_1 \cdots \gamma_k$ is the cycle representation of $\pi$ and every cycle $\gamma_i$ has length $\ell_i$ then
the multiset $\ct(\pi) := \msl \ell_1, \ldots, \ell_k \msr$ is the {\em cycle type} of $\pi$.
Note that in this situation we have 
\begin{equation*} 
\ord(\pi) = \lcm(\ell_1, \ldots, \ell_k).
\end{equation*}
 The following lemma is well known, see e.g. \cite{Cam10}:
\begin{lemma}\label{lemma-conjugate}
For $\pi, \rho \in Sym(n)$ we have $\ct(\pi) = \ct(\rho)$ if and only if there is a $\sigma \in \Sym(n)$
such that $\pi = \sigma^{-1} \rho \sigma$.
\end{lemma}
Also the following lemma seems to be folklore. For completeness we give a proof.

\begin{lemma}\label{lemmasplit}
Let $x \in \mathbb{N}$ and $\gamma$ be a single cycle of length $\ell$. Then the cycle representation of $\gamma^x$ consists of $\gcd(x,\ell)$ many disjoint cycles of length $\ell/\gcd(x,\ell)$.
\end{lemma}

\begin{proof}
Let us first consider the case where $\gcd(x,\ell)=1$. Then there is a $y \in \mathbb{N}$ with $xy \equiv 1 \bmod \ell$.
If $\gamma^x$ consists of at least two cycles of length strictly smaller than $\ell$, 
then the same holds for every power of $\gamma^x$. This contradicts $(\gamma^x)^y = \gamma^{xy} = \gamma$.
This shows the statement of the lemma for the case $\gcd(x,\ell)=1$. 

For the general case let $m = \gcd(x,\ell), k = \ell/m$ and $z = x/m$. Then we can write the cycle  $\gamma$ as 
$\gamma = (a_0,\dots,a_{mk-1})$
for some pairwise different $a_i \in [n]$. 
For all $i \in [0,m-1]$ and $d \in [0,k-1]$ we have $a_{dm+i} \gamma^m = a_{(d+1)m+i}$ where all arithmetics in the indices
is done modulo $\ell =  mk$. We obtain
\begin{displaymath}
\gamma^x = (\gamma^m)^z = \prod_{i=0}^{m-1} (a_i,a_{m+i},a_{2m+i},\dots,a_{(k-1)m+i})^z.
\end{displaymath}
Since $\gcd(z,k)=1$ we obtain from the above case  $\gcd(x,\ell)=1$ that
\begin{displaymath}
(a_i,a_{m+i},a_{2m+i},\dots,a_{(k-1)m+i})^z
\end{displaymath}
is a cycle of length $k$. Hence, $\gamma^x$ splits into $m$ disjoint cycles of length $k$. \qed
\end{proof}
For integers $1 \leq i < j \leq n$ we denote with 
 $\llbracket i,j \rrbracket$ the cycle $(i,i+1,\dots,j) \in \Sym(n)$. We also use $\llbracket i \rrbracket$ instead of $\llbracket 1,i \rrbracket$ for $2 \leq i \leq n$. 
 
 We will consider the following two computational problems in this paper:
\begin{problem} {\sf CycleType} is the following problem: \label{cycleType}
\begin{compactitem}
\item input: $\pi_1,\dots,\pi_m,\rho \in \Sym(n)$
\item question: Is there an element $\pi \in \langle \pi_1,\dots,\pi_m \rangle$ such that $\ct(\pi) = \ct(\rho)$?
\end{compactitem}
\end{problem}
\begin{problem}
\textsf{FixpointFree} is the following problem:
\begin{compactitem}
\item input: $\pi_1,\dots,\pi_m \in \Sym(n)$
\item question: Does $\fpf(n) \cap \langle \pi_1,\dots,\pi_m \rangle \neq \emptyset$ hold?
\end{compactitem}
\end{problem}
Note that the unary encoding of $n$ (from $\Sym(n)$) is implicitly part of the inputs for {\sf CycleType} and \textsf{FixpointFree}.
It is easy to see that {\sf CycleType} and \textsf{FixpointFree} are in {\sf NP}: on input 
$\pi_1,\dots,\pi_m,\rho \in \mathsf{Sym}(n)$ we guess a permutation $\pi\in \mathsf{Sym}(n)$
and then check in polynomial time whether (i) $\pi \in \langle \pi_1,\dots,\pi_m \rangle$ \cite{BaLuSe87} and 
 (ii) $\ct(\pi) = \ct(\rho)$ (resp., $\pi \in \fpf(n)$).
 
For a given number $k$ we denote with {\sf CycleType}$(k)$ the restriction of
{\sf CycleType} where $m \leq k$ holds. In other words, the input permutation group is generated by $k$ permutations.
Moreover, if the input permutations $\pi_1, \ldots, \pi_k$ pairwise commute, then we write {\sf CycleType}$(\mathsf{ab},k)$
(\textsf{ab} stands for ``abelian''). Analogous restrictions are defined for \textsf{FixpointFree}.

\section{Cycle type in cyclic permutation groups}

In this section, we study the problem {\sf CycleType}(1), i.e., {\sf CycleType} for cyclic permutation groups.
Let us fix a symmetric group $\Sym(n)$. We assume that $n$ is given in unary encoding for the following.
Let $\mathsf{P}_n$ be the set of all primes in $[n]$. One 
can easily produce a list $p_1 < p_2 < \cdots < p_r$ of all those primes in logspace.
For this, one only needs the fact that integer division for unary encoded integers can be done in logspace (actually,
integer division of binary encoded integers can be also done in logspace \cite{HesseAB02} but this is not needed here).
We will only consider numbers where all prime divisors are from $\mathsf{P}_n$.
For such a number $a$ we denote with $\pe(a)$ (for prime exponents) the tuple $(e_1,\ldots, e_r)$
such that $a = \prod_{i=1}^r p_i^{e_i}$ is the  prime factorizaton of $a$. We will represent the exponents
$e_i$ in unary notation. From the unary representation of the number $a \in [n]$ one can easily
compute in logspace the tuple $\pe(a)$. We need the following fact:

\begin{lemma} \label{lemma-pe-ord}
From a given permutation $\pi \in \Sym(n)$ one can compute in logspace the tuple $\pe(\ord(\pi))$.
\end{lemma}

\begin{proof}
Assume that the cycle representation $\pi = \gamma_1 \gamma_2 \cdots \gamma_k$ is given.
Let $\ell_i \in [n]$ be the length of the cycle $\gamma_i$. We then compute in logspace the tuple $\pe(\ell_i) = (e_{i,1}, \ldots, e_{i,r})$.
Since $\ord(\pi) = \lcm(\ell_1, \ell_2, \ldots, \ell_k)$ we have
$\pe(\ord(\pi)) = (e_1, \ldots, e_r)$ with $e_i = \max\{ e_{1,i}, \ldots, e_{k,i}\}$. Clearly, these exponents $e_i$ can be 
computed in logspace. \qed
\end{proof}

\begin{lemma} \label{cycleTypeM1Lemma1}
For given permutations $\pi, \rho \in \Sym(n)$ one can check in logspace, whether $\ord(\rho)\mid\ord(\pi)$ holds.
\end{lemma}

\begin{proof}
Let $\pe(\ord(\rho)) = (e_1, \ldots, e_r)$ and $\pe(\ord(\pi)) = (e'_1, \ldots, e'_r)$. Then $\ord(\rho) \mid \ord(\pi)$  if and only
if $e_i \leq  e'_i$ for all $i \in [r]$. Therefore, the statement of the lemma follows from Lemma~\ref{lemma-pe-ord}. \qed
\end{proof}

\begin{lemma}\label{cycleTypeM1Lemma2}
There is a logspace algorithm with the following specification:
\begin{compactitem}
\item input: $\pi,\rho \in \Sym(n)$ such that $\ord(\rho) \mid \ord(\pi)$ and $a \in [n]$.
\item output: $a\pi^{d} \in [n]$ where $d = \ord(\pi)/\ord(\rho)$
\end{compactitem}
\end{lemma}

\begin{proof}
By Lemma~\ref{lemma-pe-ord} we can produce in logspace the tuples 
$\pe(\ord(\rho)) = (e_1, \ldots, e_r)$ and $\pe(\ord(\pi)) = (e'_1, \ldots, e'_r)$. Since 
$\ord(\rho) \mid \ord(\pi)$ we have $e_i \leq e'_i$ for all $i \in [r]$.
We then have $\pe(d) = (f_1, \ldots, f_r)$ with $f_i = e'_i-e_i$ and this tuple
can be also produced in logspace.

Let $\gamma_1 \gamma_2 \cdots \gamma_k$ be the cycle representation of $\pi$.
We then compute in logspace the length $\ell \in [n]$ of the unique cycle $\gamma_i$ that contains $a \in [n]$.
We have $a \pi^d = a \pi^{d \bmod \ell}$.
Since all primes $p_i$ and exponents $f_i$ are given in unary notation, we can compute in logspace the value 
$d \bmod \ell$ by going over the prime factorization $\prod_{i=1}^r p_i^{f_i}$ and making
$\sum_{i=1}^r f_i$ many multiplications modulo $\ell$. Once $d \bmod \ell$ is computed, we can
finally compute $a \pi^{d \bmod \ell}$ in logspace. \qed
\end{proof}

\begin{lemma}\label{cycleTypeM1Lemma3}
Let $\pi,\rho \in \Sym(n)$. Then the following holds:
\begin{compactitem}
\item If $\ct(\pi) = \ct(\rho)$ then $\ord(\pi) = \ord(\rho)$.
\item For all $i \in \mathbb{N}$ we have $\ord(\pi) = \ord(\pi^i)$ if and only if $\ct(\pi) = \ct(\pi^i)$.
\end{compactitem}
\end{lemma}
\begin{proof}
For the first statement note that if $\msl \ell_1, \ell_2, \ldots, \ell_k \msr$ is the common cycle type of $\pi$ and $\rho$ then
$\ord(\pi) = \lcm(\ell_1, \ell_1, \ldots, \ell_k) = \ord(\rho)$.
 Therefore we only have to show that if $\ord(\pi) = \ord(\pi^i)$ then $\ct(\pi) = \ct(\pi^i)$. Let $\pi = \gamma_1 \cdots \gamma_k$ be the 
 cycle representation of $\pi$.
 Then we have $\pi^i = \gamma_1^i \cdots \gamma_k^i$. Since $\ord(\pi) = \ord(\pi^i)$ we obtain $\gcd(\ord(\pi),i)=1$. Because of $\ord(\gamma_j) \mid \ord(\pi)$ we get $\gcd(\ord(\gamma_j),i)=1$ for all $j \in [k]$. By Lemma~\ref{lemmasplit},
 $\gamma_j$ and $\gamma_j^i$ are cycles of the same length and thus $\pi$ and $\pi^i$ have the same cycle type. \qed
\end{proof}

\begin{theorem}\label{cycleTypeM1Theorem}
$\mathsf{CycleType}(1)$ is in {\sf L}.
\end{theorem}
\begin{proof}
Let $\pi,\rho \in \Sym(n)$ be the two input permutations of \textsf{CycleType}(1). It is asked whether there is a $q \in \mathbb{N}$ such that 
$\ct(\pi^q) = \ct(\rho)$. By Lemma~\ref{cycleTypeM1Lemma1} we can check in logspace whether $\ord(\rho) \mid \ord(\pi)$
holds. If this is not the case, then by the first statement of Lemma~\ref{cycleTypeM1Lemma3} there is no $q$ such that $\ct(\pi^q) = \ct(\rho)$
 and we can immediately reject. Let us now assume that $\ord(\rho) \mid \ord(\pi)$ and let $d = \ord(\pi)/\ord(\rho)$ in the following.
Note that $\ord(\pi^d) = \ord(\rho)$.

\begin{ownclaim} \label{claim-logspace}
There is a $q \in \mathbb{N}$ such that $\ct(\pi^q) = \ct(\rho)$
if and only if $\ct(\pi^d) = \ct(\rho)$.
\end{ownclaim}
\emph{Proof of Claim~\ref{claim-logspace}.}
The direction from right to left is trivial. Hence, let us assume that there is a $q$ such that
$\ct(\pi^q) = \ct(\rho)$.
By Lemma~\ref{cycleTypeM1Lemma3}, we have $\ord(\pi^q) = \ord(\rho)$. We get $\ord(\pi^d) = \ord(\rho) = \ord(\pi^q)$. 
Since $\langle \pi \rangle$ has exactly one subgroup of order $\ord(\rho)$ it follows that $\langle \pi^q \rangle = \langle \pi^d \rangle$. 
Let $\pi^q = (\pi^d)^i$ for $i \in \mathbb{N}$. Since $\ord(\pi^d) = \ord(\pi^q) = \ord((\pi^d)^i)$, 
the second statement of Lemma~\ref{cycleTypeM1Lemma3} implies that $\pi^q$ and $\pi^d$ (and hence $\rho$ and $\pi^d$) have
the same cycle type. This shows Claim~\ref{claim-logspace}.

\smallskip
\noindent
By Claim~\ref{claim-logspace}, it suffices to check in logspace whether $\ct(\pi^d) = \ct(\rho)$.
By Lemma~\ref{cycleTypeM1Lemma2} we can compute in logspace
the pointwise representation and hence the cycle representation of $\pi^d$. From the cycle representation
of a permutation we can of course compute in logspace the cycle type.
\qed
\end{proof}

\section{Cycle type in the 2-generated abelian case}

In this section we show that \textsf{CycleType} becomes \textsf{NP}-complete if the input permutation group is abelian and generated
by two elements.

\begin{theorem}\label{theoremcycletypenpcomplete}
$\mathsf{CycleType}(\mathsf{ab},2)$ is {\sf NP}-complete.
\end{theorem}

\begin{proof}
Since \textsf{CycleType} is  in {\sf NP} (see the remark at the end of Section~\ref{sec-perm}), it remains to show {\sf NP}-hardness. For this we 
exhibit a logspace reduction from \textsf{X3HS} (exact 3-hitting set), which is the following problem:
\begin{compactitem}
\item Input: a finite set $S$ and a set $\mathcal{B} \subseteq 2^S$ of subsets of $S$ all of size $3$.
\item Question: Is there a subset $T \subseteq S$ such that $|T \cap C| = 1$ for all $C \in \mathcal{B}$?
\end{compactitem}
Note that \textsf{X3HS} is the same problem as positive \textsf{1-in-3-SAT}, which is a well-known {\sf NP}-complete problem; 
see  \cite{gareyjohnson} for more details.

Let $S$ be a finite set and $\mathcal{B} \subseteq 2^S$ be a set of subsets of $S$ all of size $3$. W.l.o.g. assume that $S = [n]$ and let $\mathcal{B} = \{C_1,\dots,C_m\}$. Let $p_1 < \dots < p_{2n}$ be the first $2n$ primes with $p_1 > 3$. Moreover let $q_1 < \dots < q_m$ be the next $m$ primes with $p_{2n} < q_1$. We associate $i \in S$ with the prime $p_i$ and $C_j \in \mathcal{B}$ with the prime $q_j$. We will work with the group
\begin{displaymath}
G = \prod_{i=1}^n \Sym(p_ip_{n+i}) \times \prod_{j=1}^m \Sym(p_n^3q_j)^6
\end{displaymath}
which naturally embedds into $\Sym(N)$ for $N = \sum_{i=1}^n p_ip_{n+i} + 6\sum_{j=1}^m p_n^3q_j$.
Let $f : G \to \Sym(N)$ be this embedding. 
When we talk of the cycle type of an element $g \in G$, we always refer to the cycle type of the permutation $f(g) \in \Sym(N)$. If
$g = (\pi_1, \ldots, \pi_n, \rho_1, \ldots, \rho_{6m}) \in G$, then this cycle type is obtained by taking the disjoint union (of multisets)
of the cycle types of all the $\pi_i$ and $\rho_j$. 

 For $j \in [m]$ we define $r_j = q_j \cdot \prod_{i \in C_j} p_i$. Moreover for $j \in [m]$ and  all $d \in [6]$ we define the number $s_{j,d} \in [0,r_j-1]$ as the smallest positive integer satisfying the following congruences in which we assume $C_j = \{i_1,i_2,i_3\}$ with $i_1 < i_2 < i_3$:
\begin{align*}
s_{j,1}&\equiv -1 \bmod p_{i_1} && s_{j,2} \equiv 0 \bmod p_{i_1} && s_{j,3} \equiv 0 \bmod p_{i_1}\\
s_{j,1}&\equiv 0  \bmod p_{i_2} && s_{j,2} \equiv -1 \bmod p_{i_2} && s_{j,3} \equiv 0 \bmod p_{i_2}\\
s_{j,1}&\equiv 0 \bmod p_{i_3} && s_{j,2} \equiv 0 \bmod p_{i_3} && s_{j,3} \equiv -1 \bmod p_{i_3}\\
s_{j,1}&\equiv 1 \bmod q_j && s_{j,2} \equiv 1 \bmod q_j && s_{j,3} \equiv 1 \bmod q_j \\[2mm]
s_{j,4}&\equiv -1 \bmod p_{i_1} && s_{j,5} \equiv -3 \bmod p_{i_1} && s_{j,6} \equiv -2 \bmod p_{i_1}\\
s_{j,4}&\equiv -2 \bmod p_{i_2} && s_{j,5} \equiv -1 \bmod p_{i_2} && s_{j,6} \equiv -3 \bmod p_{i_2}\\
s_{j,4}&\equiv -3 \bmod p_{i_3} && s_{j,5} \equiv -2 \bmod p_{i_3} && s_{j,6} \equiv -1 \bmod p_{i_3}\\
s_{j,4}&\equiv 1 \bmod q_j && s_{j,5} \equiv 1 \bmod q_j && s_{j,6} \equiv 1 \bmod q_j
\end{align*}
Moreover, we define the number $t_j \in [0, r_j-1]$ as the smallest positive integer satisfying
\begin{equation*}
t_j\equiv 1 \bmod p_{i_a} \text{ for all } a \in [3] \text{ and } t_j\equiv 0 \bmod q_j.
\end{equation*}
We define the input group elements $\rho, \pi_1, \pi_2 \in G$ as follows, where $i$ ranges over $[n]$, $j$ ranges over $[m]$ and $i_1 < i_2 < i_3$ are the elements of $C_j$ (recall that $\llbracket m \rrbracket$ denotes the cycle $(1,2,\ldots,m)$):
\allowdisplaybreaks
\begin{align*} 
\rho &= (\zeta_1,\dots,\zeta_n,\eta_1,\dots,\eta_m)\\
\zeta_i &= \llbracket p_ip_{n+i} \rrbracket\\
\eta_j &= (\llbracket r_j \rrbracket^{p_{i_1}p_{i_2}p_{i_3}}, \llbracket r_j \rrbracket^{p_{i_1}}, \llbracket r_j \rrbracket^{p_{i_2}}, \llbracket r_j \rrbracket^{p_{i_3}}, \llbracket r_j \rrbracket, \llbracket r_j \rrbracket) \\[1mm]
\pi_1 &= (\alpha_1,\dots,\alpha_n,\beta_1,\dots,\beta_m)\\
\alpha_i &= \llbracket p_ip_{n+i} \rrbracket\\
\beta_j &= (\llbracket r_j \rrbracket^{s_{j,1}}, \llbracket r_j \rrbracket^{s_{j,2}}, \llbracket r_j \rrbracket^{s_{j,3}}, \llbracket r_j \rrbracket^{s_{j,4}}, \llbracket r_j \rrbracket^{s_{j,5}}, \llbracket r_j \rrbracket^{s_{j,6}}) \\[1mm]
\pi_2 &= (\gamma_1,\dots,\gamma_n,\delta_1,\dots,\delta_m)\\
\gamma_i &= \mathrm{id}\\
\delta_j &= (\llbracket r_j \rrbracket^{t_j}, \llbracket r_j \rrbracket^{t_j}, \llbracket r_j \rrbracket^{t_j}, \llbracket r_j \rrbracket^{t_j}, \llbracket r_j \rrbracket^{t_j}, \llbracket r_j \rrbracket^{t_j})
\end{align*}
Note that $\pi_1$ and $\pi_2$ commute. 

We will show there are $x_1,x_2 \in \mathbb{N}$ such that $\ct(\rho) = \ct(\pi_1^{x_1}\pi_2^{x_2})$ if and only if there is a subset $T \subseteq S$ such that $|T \cap C_j| = 1$ for all $j \in [m]$.

First suppose that there are $x_1,x_2 \in \mathbb{N}$ with $\ct(\rho) = \ct(\pi_1^{x_1}\pi_2^{x_2})$. We define
\begin{equation} \label{eq-T}
T = \{i \in [n] \mid x_2 \not \equiv 0 \bmod p_i\}.
\end{equation}
\vspace*{-6mm}
\begin{ownclaim}\label{claimx1notequiv0}
For all $i \in [n]$ and $j \in [m]$ we have $x_1 \not \equiv 0 \bmod p_i$, $x_1 \not \equiv 0 \bmod p_{n+i}$ and $x_1 \not \equiv 0 \bmod q_j$.
\end{ownclaim}
\emph{Proof of Claim~\ref{claimx1notequiv0}.} The claim follows from Lemma~\ref{lemmasplit} and the following facts:
\begin{compactitem}
\item $\zeta_i$ and $\alpha_i$ are cycles of length $p_ip_{n+i}$.
\item $\pi_2$ does not contain any cycle whose length is a multiple of $p_{n+i}$.
\item $t_j \equiv 0 \bmod q_j$ and hence $\pi_2$ also does not contain any cycle whose length is a multiple of $q_j$.
\item $\rho$ and $\pi_1$ both contain $6$ pairwise disjoint permutations of the form $\llbracket r_j \rrbracket^z$, where $z$ is not a multiple of $q_j$. \qed
\end{compactitem}

\begin{ownclaim}\label{claimcjexactlyonea}
For all $C_j = \{i_1,i_2,i_3\} \in \mathcal{B}$ there is a (necessarily unique) $a \in [3]$ such that $x_2 \not \equiv 0 \bmod p_{i_a}$ and $x_2 \equiv 0 \bmod p_{i_b}$ for all $b \in [3] \setminus \{a\}$.
\end{ownclaim}
\emph{Proof of Claim~\ref{claimcjexactlyonea}.} 
Let $j \in [m]$ and assume $C_j = \{i_1,i_2,i_3\}$ with $i_1 < i_2 < i_3$. Consider $\eta_j$. By Lemma~\ref{lemmasplit} $\llbracket r_j \rrbracket^{p_{i_1}p_{i_2}p_{i_3}}$ consists of $p_{i_1}p_{i_2}p_{i_3}$ cycles of length $q_j$ and these are the only cycles of length $q_j$ in $\rho$.
Hence, $\beta_j^{x_1}\delta_j^{x_2}$ must contain exactly $p_{i_1}p_{i_2}p_{i_3}$ cycles of length $q_j$. By Lemma~\ref{lemmasplit} this can only be achieved if there is a unique $a \in [6]$ such that $x_1s_{j,a} + x_2t_j \equiv 0 \bmod p_{i_c}$ for all $c \in [3]$. Also note that $x_1s_{j,b} + x_2t_j \equiv x_1 \not \equiv 0 \bmod q_j$ for all $b \in [6]$ by Claim~\ref{claimx1notequiv0} and $x_2t_j \equiv x_2 \bmod p_{i_c}$ for all $c \in [3]$. 

We want to show that $a \in [3]$. In order to get a contradiction,
suppose that $a \in \{4,5,6\}$. The congruence $x_1s_{j,a} + x_2t_j \equiv 0 \bmod p_{i_c}$ gives us $x_2 \equiv -x_1s_{j,a} \bmod p_{i_c}$ for all $c \in [3]$. Then, for all $b \in [3] \setminus \{a-3\}$ we have
\begin{displaymath}
x_1s_{j,b} + x_2t_j \equiv x_1s_{j,b} - x_1s_{j,a} \equiv x_1(-1 - s_{j,a}) \not \equiv 0 \bmod p_{i_b},
\end{displaymath}
where $x_1 \not \equiv 0 \bmod p_{i_b}$ by Claim~\ref{claimx1notequiv0} and 
$-1 - s_{j,a}  \not \equiv 0 \bmod p_{i_b}$ since 
$s_{j,a} \not\equiv -1 \bmod  p_{i_b}$ 
for $b \neq a-3$ (also note that $p_{i_b} > 2$). Similarly, for all $b \in [3] \setminus \{a-3\}$ we get
\begin{displaymath}
x_1s_{j,3+b} + x_2t_j \equiv x_1s_{j,3+b} - x_1s_{j,a} \equiv x_1(s_{j,3+b} - s_{j,a}) \not \equiv 0 \bmod p_{i_b},
\end{displaymath}
where as above $x_1 \not \equiv 0 \bmod p_{i_b}$ by Claim~\ref{claimx1notequiv0} and $s_{j,3+b} - s_{j,a} \not \equiv 0 \bmod p_{i_b}$ since $a \neq 3+b$ and $s_{j,a} \not \equiv s_{j,3+b} \bmod p_{i_c}$ for all $c \in [3]$. 

Moreover, for all $b \in [3] \setminus \{a-3\}$ and all $c \in [3] \setminus \{b\}$ we have
\begin{align*}
x_1s_{j,b} + x_2t_j & \equiv x_1s_{j,b}-x_1s_{j,a} \equiv -x_1s_{j,a} \not \equiv 0 \bmod p_{i_c} \text{ and } \\
x_1s_{j,3+b} + x_2t_j & \equiv x_1s_{j,3+b}-x_1s_{j,a} \equiv x_1(s_{j,3+b}-s_{j,a}) \not \equiv 0 \bmod p_{i_c}.
\end{align*}
Finally, for all $b \in [6]$ we have $x_1s_{j,b} + x_2t_j \not \equiv 0 \bmod q_j$
as pointed out above. Taken together, these congruences yield for all $b \in [3] \setminus \{a-3\}$:
\begin{displaymath}
\gcd(x_1s_{j,b} + x_2t_j  , r_j) = 
\gcd(x_1s_{j,3+b} + x_2t_j, r_j) = 1.
\end{displaymath}
Hence, by Lemma~\ref{lemmasplit}, $\beta_j^{x_1}\delta_j^{x_2}$ contains at least $4$ cycles of length $r_j$. However $\eta_j$ 
contains only $2$ cycles of length $r_j$ and $\rho$ does not contain any other cycles of length $r_j$, which gives us a contradiction. Thus we obtain $a \in [3]$ and by this
\begin{displaymath}
x_2 \equiv -x_1s_{j,a} \equiv x_1 \not \equiv 0 \bmod p_{i_a},
\end{displaymath}
where $x_1 \not \equiv 0 \bmod p_{i_a}$ holds by Claim~\ref{claimx1notequiv0}. Moreover, for all $b \in [3] \setminus \{a\}$ we obtain
\begin{displaymath}
x_2 \equiv -x_1s_{j,a} \equiv 0 \bmod p_{i_b}.
\end{displaymath}
This shows Claim~\ref{claimcjexactlyonea}.
\qed

\medskip
\noindent
We can now show that $|T \cap C_j| = 1$ for all $j \in [m]$. Let $j \in [m]$. By Claim~\ref{claimcjexactlyonea} there is a unique $i \in C_j$ such that $x_2 \not \equiv 0 \bmod p_i$. Thus $i \in T$ by \eqref{eq-T}. Moreover for all $h \in C_j \setminus \{i\}$ we have $x_2 \equiv 0 \bmod p_h$ by Claim~\ref{claimcjexactlyonea} and hence $h \notin T$. Thus, we get $|T \cap C_j| = 1$.

For the other direction, suppose there is a subset $T \subseteq [n]$ such that $|T \cap C_j| = 1$ for all $j \in [m]$. We define $x_1 = 1$ and $x_2$ as the smallest positive integer satisfying the congruences
\begin{displaymath}
x_2 \equiv \begin{cases}
1 \bmod p_i & \text{ if } i \in T\\
0 \bmod p_i & \text{ if } i \not \in T
\end{cases}
\end{displaymath}
for all $i \in [n]$. Since $x_1 = 1$, $\rho$ and $\pi_1^{x_1} \pi_2^{x_2}$ both contain a unique cycle of length $p_i p_{n+i}$ for all $i \in [n]$.
All other cycles in $\rho$ and $\pi_1^{x_1} \pi_2^{x_2}$ result from powers of $\llbracket r_j \rrbracket$ for some $j \in [m]$.
Consider a $j \in [m]$ and let $C_j = \{i_1,i_2,i_3\}$ with $i_1 < i_2 < i_3$. By Lemma~\ref{lemmasplit}, $\eta_j$ consists of 
\begin{compactenum}[(i)]
\item $p_{i_1}p_{i_2}p_{i_3}$ cycles of length $q_j$, 
\item $p_{i_1}$ cycles of length $p_{i_2}p_{i_3}q_j$, 
\item $p_{i_2}$ cycles of length $p_{i_1}p_{i_3}q_j$, 
\item $p_{i_3}$ cycles of length $p_{i_1}p_{i_2}q_j$ and 
\item $2$ cycles of length $r_j$.
\end{compactenum} 

\medskip
\noindent
We have to show that 
\begin{align*}
 \beta_j\delta_j^{x_2} =  (&\llbracket r_j \rrbracket^{s_{j,1}+x_2t_j}, \llbracket r_j \rrbracket^{s_{j,2}+x_2t_j}, \llbracket r_j \rrbracket^{s_{j,3}+x_2t_j}, \\ &\llbracket r_j \rrbracket^{s_{j,4}+x_2t_j}, \llbracket r_j \rrbracket^{s_{j,5}+x_2t_j}, \llbracket r_j \rrbracket^{s_{j,6}+x_2t_j})
\end{align*}
contains the same cycle lengths with the same
multiplicities as in (i)--(v). Note that $s_{j,d} + x_2t_j \equiv 1 \bmod q_j$ for all $d \in [6]$. Let $a \in [3]$ be the unique element with $i_a \in T$. Then $x_2 \equiv 1 \bmod p_{i_a}$ and  $x_2 \equiv 0 \bmod p_{i_b}$ for all $b \in [3] \setminus \{a\}$. We obtain
\begin{align*}
s_{j,a}+x_2t_j &\equiv -1 + 1 \equiv 0 \bmod p_{i_a} \text{ and }\\
s_{j,a}+x_2t_j &\equiv 0 + 0 \equiv 0 \bmod p_{i_b} \text{ for all } b \in [3] \setminus \{a\}.
\end{align*}
By Lemma~\ref{lemmasplit}, $\llbracket r_j \rrbracket^{s_{j,a}+x_2t_j}$ consists of $p_{i_1}p_{i_2}p_{i_3}$ cycles of length $q_j$. Moreover
\begin{align*}
s_{j,3+a}+x_2t_j &\equiv -1 + 1 \equiv 0 \bmod p_{i_a} \text{ and } \\
s_{j,3+a}+x_2t_j &\equiv s_{j,3+a} + 0 \not \equiv 0 \bmod p_{i_b} \text{ for all } b \in [3] \setminus \{a\}
\end{align*}
(for the second point we use the fact that all primes $p_i$ are larger than $3$).
By Lemma~\ref{lemmasplit}, $\llbracket r_j \rrbracket^{s_{j,3+a}+x_2t_j}$ consists of $p_{i_a}$ cycles of length $q_j\prod_{b \in [3] \setminus \{a\}} p_{i_b}$. For all $b \in [3] \setminus \{a\}$ we have 
\begin{align*}
s_{j,b}+x_2t_j &\equiv 0 + 1 \equiv 1 \bmod p_{i_a}, \\
s_{j,b}+x_2t_j &\equiv s_{j,b} + 0 \equiv -1 \bmod p_{i_b} \text{ and } \\
s_{j,b}+x_2t_j &\equiv s_{j,b} + 0 \equiv 0 \bmod p_{i_c}, \text{where } \{c\} = [3] \setminus \{a,b\}.
\end{align*}
By Lemma~\ref{lemmasplit}, $\llbracket r_j \rrbracket^{s_{j,b}+x_2t_j}$ consists of $p_{i_c}$ cycles of length $q_jp_{i_a}p_{i_b}$ with $\{c\} = [3] \setminus \{a,b\}$. Finally, for all $b \in [3] \setminus \{a\}$ we have
\begin{align*}
s_{j,3+b}+x_2t_j &\equiv s_{j,3+b} + 1 \not \equiv 0 \bmod p_{i_a} \text{ and }\\
s_{j,3+b}+x_2t_j &\equiv s_{j,3+b} + 0 \not \equiv 0 \bmod p_{i_c} \text{ for all } c \in [3] \setminus \{a\} .
\end{align*}
Hence, $\llbracket r_j \rrbracket^{s_{j,3+b}+x_2t_j}$ is a single cycle of length $r_j$. This shows that
$\ct(\eta_j)=\ct(\beta_j\delta_j^{x_2})$ and concludes the proof of the theorem.
\qed
\end{proof}
The construction from the previous proof yields the following additional result:

\begin{corollary} \label{coro-cosets}
The following problem is  {\sf NP}-complete:
\begin{compactitem}
\item input: $\rho, \pi_1,\pi_2\in \Sym(n)$ such that $\pi_1$ and $\pi_2$ commute
\item question: Is there is a $\pi \in \pi_1 \langle \pi_2 \rangle$ such that $\ct(\rho) = \ct(\pi)$?
\end{compactitem}
\end{corollary}
\begin{proof}
The instance  $\rho,\pi_1,\pi_2$ of $\mathsf{CycleType}(\mathsf{ab},2)$ that we constructed in the proof of Theorem~\ref{theoremcycletypenpcomplete}
has the property that
there are $x_1, x_2 \in \mathbb{N}$ such that $\rho$ and $\pi_1^{x_1}\pi_2^{x_2}$ have the same cycle type
if and only if there is $x_2 \in \mathbb{N}$ such that $\rho$ and $\pi_1\pi_2^{x_2}$ have the same cycle type.
This yields the corollary.
\qed
\end{proof}
Whereas it can be decided in logspace whether a cyclic permutation group $\langle \pi_1 \rangle$ contains 
a permutation with a given cycle type (Theorem~\ref{cycleTypeM1Theorem}), the same problem for cosets of cyclic permutation groups is \textsf{NP}-complete (Corollary~\ref{coro-cosets}).

\section{Fixpoint freeness in the 2-generated abelian case}

Our main result for the problem $\mathsf{FixpointFree}$ is:
\begin{theorem}\label{theoremfpfnpcomplete}
$\mathsf{FixpointFree}(\mathsf{ab},2)$ is ${\sf NP}$-complete.
\end{theorem}

\begin{proof}
We give a logspace reduction from \textsf{3-SAT} (the satisfiability problem for 
conjunctions of clauses, where every clause consists
of exactly three literals and a literal is either a boolean variable $x$ or a negated boolean variable $\bar x$).  
For this take a finite set of variables $X = \{x_1,\dots,x_n\}$ and a set of clauses $\mathcal{C} = \{C_1,\ldots,C_m\}$.
Every $C_j \in \mathcal{C}$ is a set of three literals.  When we write $C_j$ as 
$C_j = \{\tilde{x}_{i_1},\tilde{x}_{i_2},\tilde{x}_{i_3}\}$, every $\tilde{x}_{i_k}$ is either $x_{i_k}$ or $\bar x_{i_k}$ and
 we always assume that $i_1 < i_2 < i_3$. 
 A truth assignment $\sigma : X \to \{0,1\}$ is implicitly extended to all literals by setting $\sigma(\bar{x}_i) = 1 - \sigma(x_i)$.

Let $p_1,\dots,p_n,\bar{p}_1,\dots,\bar{p}_n$ be the first $2n$ primes. We associate the positive literal $x_i$ with $p_i$ and the negative literal $\bar{x}_i$ with $\bar{p}_i$ and define
\begin{displaymath}
\tilde{p}_{i} = \begin{cases}
p_{i} & \text{if } \tilde{x}_{i} = x_{i}, \\
\bar{p}_{i} & \text{if } \tilde{x}_{i} = \bar{x}_{i} .
\end{cases}
\end{displaymath}
For the clause $C_j = \{\tilde{x}_{i_1},\tilde{x}_{i_2},\tilde{x}_{i_3}\}$ define $r_j = \tilde{p}_{i_1}\tilde{p}_{i_2}\tilde{p}_{i_3}$.
Moreover,  for all $i \in [n],l \in [p_i-1]$ and $k \in [\bar{p}_i-1]$ let $s_{i,l,k}$ be the unique 
number in $[p_i\bar{p}_i-1]$ with
\begin{equation*}
s_{i,l,k}\equiv l \bmod p_i \quad \text{ and } \quad
s_{i,l,k}\equiv k \bmod \bar{p}_i.
\end{equation*}
We will work with the group
\begin{equation*}
G   = \prod_{i=1}^n \left(\Sym(p_i) \times \Sym(\bar{p}_i) \times \Sym(p_i\bar{p}_i)^{(p_i-1)(\bar{p}_i-1)+1}\right) \times 
 \prod_{j=1}^m \Sym(r_j) .
\end{equation*}
The group $G$ naturally embeds into $\Sym(N)$ for
\begin{displaymath}
N = \sum_{i=1}^n (p_i + \bar{p}_i + p_i\bar{p}_i((p_i-1)(\bar{p}_i-1)+1)) + \sum_{j=1}^m r_j.
\end{displaymath}
Now we define the input permutations $\pi_1$ and $\pi_2$ as follows, where $i$ ranges over $[n]$, $l$ ranges over $[p_i-1]$, $k$ ranges over $[\bar{p}_i-1]$ and $j$ ranges over $[m]$:
\begin{align*}
\pi_1&= (\alpha_1,\dots,\alpha_n,\beta_1,\dots,\beta_m) \text{ with}  \\
\alpha_i&= (\alpha_{i,1},\alpha_{i,2},\alpha_{i,3},\alpha_{i,1,1},\dots,\alpha_{i,p_i-1,\bar{p}_i-1}) \\
\alpha_{i,1}&= \llbracket p_i \rrbracket \\
\alpha_{i,2}&= \llbracket \bar{p}_i \rrbracket \\
\alpha_{i,3}&= \mathrm{id} \\
\alpha_{i,l,k}&= \llbracket p_i\bar{p}_i \rrbracket^{s_{i,l,k}} \\
\beta_j&= \mathrm{id}  \\[1mm]
 \pi_2 & = (\gamma_1,\dots,\gamma_n,\delta_1,\dots,\delta_m) \text{ with} \\
 \gamma_i & = (\gamma_{i,1},\gamma_{i,2},\gamma_{i,3},\gamma_{i,1,1},\dots,\gamma_{i,p_i-1,\bar{p}_i-1}) \\
 \gamma_{i,1} & = \gamma_{i,2} = \mathrm{id} \\
\gamma_{i,3} &= \gamma_{i,l,k} = \llbracket p_i\bar{p}_i \rrbracket \\
\delta_j &= \llbracket r_j \rrbracket
\end{align*}
Note that $\pi_1$ and $\pi_2$ commute. We will show that $\mathcal{C}$ is satisfiable if and only if there are  $z_1,z_2 \in \mathbb{N}$ such that $\pi_1^{z_1}\pi_2^{z_2} \in \fpf(n)$.

First, suppose that  there are $z_1,z_2 \in \mathbb{N}$ such that $\pi_1^{z_1}\pi_2^{z_2} \in \fpf(n)$.

\begin{ownclaim}\label{claimx1not0pifpf}
For all $i \in [n]$ we have $z_1 \not \equiv 0 \bmod p_i$ and $z_1 \not \equiv 0 \bmod \bar{p}_i$.
\end{ownclaim}
We have $\alpha_{i,1}^{z_1}\gamma_{i,1}^{z_2} = \alpha_{i,1}^{z_1} = \llbracket p_i \rrbracket^{z_1}$
and hence by Lemma~\ref{lemmasplit} we obtain $z_1 \not \equiv 0 \bmod p_i$. Analogously we obtain $z_1 \not \equiv 0 \bmod \bar{p}_i$.
\qed

\begin{ownclaim}\label{claimx2not01pibpi}
For all $i \in [n]$ we have $z_2 \equiv 0 \bmod p_i$ if and only if $z_2 \not \equiv 0 \bmod \bar{p}_i$.
\end{ownclaim}
Assume that $z_2 \equiv 0 \bmod p_i$ and $z_2 \equiv 0 \bmod \bar{p}_i$. Then we obtain by
\begin{displaymath}
\alpha_{i,3}^{z_1}\gamma_{i,3}^{z_2} = \gamma_{i,3}^{z_2} = \llbracket p_i\bar{p}_i \rrbracket^{z_2} = \mathrm{id}
\end{displaymath}
a contradiction. Now assume that $z_2 \not \equiv 0 \bmod p_i$ and $z_2 \not \equiv 0 \bmod \bar{p}_i$. Since by Claim~\ref{claimx1not0pifpf} we have $z_1 \not \equiv 0 \bmod p_i$ and $z_1 \not \equiv 0 \bmod \bar{p}_i$ we can define $l \in [p_i-1]$ and $k \in [\bar{p}_i-1]$ as the smallest positive integers satisfying the congruences
\begin{equation*}
l\equiv -z_2z_1^{-1} \bmod p_i \quad \text{ and } \quad k\equiv -z_2z_1^{-1} \bmod \bar{p}_i.
\end{equation*}
From this we obtain $s_{i,l,k} \equiv -z_2z_1^{-1} \bmod p_i\bar{p}_i$
and hence
\begin{displaymath}
\alpha_{i,l,k}^{z_1}\gamma_{i,l,k}^{z_2} = \llbracket p_i\bar{p}_i \rrbracket^{s_{i,l,k} \cdot z_1}\llbracket p_i\bar{p}_i \rrbracket^{z_2} = \llbracket p_i\bar{p}_i \rrbracket^{-z_2}\llbracket p_i\bar{p}_i \rrbracket^{z_2} = \mathrm{id},
\end{displaymath}
which is again a contradiction. This shows Claim~\ref{claimx2not01pibpi}
\qed

\begin{ownclaim}\label{claimclauseexistst}
For all $j \in [m]$ there is an $a \in [3]$ such that $z_2 \not \equiv 0 \bmod \tilde{p}_{i_a}$, where $C_j = \{\tilde{x}_{i_1},\tilde{x}_{i_2},\tilde{x}_{i_3}\}$.
\end{ownclaim}
Since we must have  $\beta_j^{z_1}\delta_j^{z_2} = \delta_j^{z_2} = \llbracket r_j \rrbracket^{z_2} \in \fpf(r_j)$ 
we must have $z_2 \not \equiv 0 \bmod r_j = \tilde{p}_{i_1}\tilde{p}_{i_2}\tilde{p}_{i_3}$. Hence, there  is an $a \in [3]$ such that $z_2 \not \equiv 0 \bmod \tilde{p}_{i_a}$.
\qed

\medskip
\noindent
We define the truth assignment $\sigma : X \to \{0,1\}$ by
\begin{displaymath}
\sigma(x_i) = \begin{cases}
1 & \text{if } z_2 \not \equiv 0 \bmod p_i\\
0 & \text{if } z_2 \equiv 0 \bmod p_i
\end{cases}
\end{displaymath}
for all $i \in [n]$ and show that every clause in $\mathcal{C}$ contains a literal that is mapped to $1$ by $\sigma$.
Let $j \in [m]$ and $C_j = \{\tilde{x}_{i_1},\tilde{x}_{i_2},\tilde{x}_{i_3}\}$. By Claim~\ref{claimclauseexistst} there is
an  $a \in [3]$ such that $z_2 \not \equiv 0 \bmod \tilde{p}_{i_a}$. If $\tilde{x}_{i_a} = x_{i_a}$, then $\tilde{p}_{i_a} = p_{i_a}$ and
$1 = \sigma(x_{i_a}) =  \sigma(\tilde{x}_{i_a})$. On the other hand,
if $\tilde{x}_{i_a} = \bar{x}_{i_a}$, then $\tilde{p}_{i_a} = \bar{p}_{i_a}$ and $z_2 \equiv 0 \bmod p_{i_a}$ by Claim~\ref{claimx2not01pibpi}. We obtain
$1 = 1 - \sigma(x_{i_a}) = \sigma(\bar{x}_{i_a}) = \sigma(\tilde{x}_{i_a})$.
Hence, $\sigma(\tilde{x}_{i_a}) = 1$ in both cases.

Vice versa suppose that there is a truth assignment $\sigma : X \to \{0,1\}$ such that
every clause in $\mathcal{C}$ contains a literal that is mapped to $1$ by $\sigma$.
We define $z_1 = 1$ and $z_2 \in \mathbb{N}$ as the smallest positive integer satisfying the congruences
\begin{equation} \label{eq-z2}
z_2 \equiv \sigma(x_i) \bmod p_i \quad \text{ and } \quad 
z_2 \equiv 1-\sigma(x_i) \bmod \bar{p}_i
\end{equation}
for all $i \in [n]$. Then $\pi_1^{z_1}\pi_2^{z_2} \in \fpf(n)$ follows from the following points, where 
$i \in [n]$, $l \in [p_i-1]$, $k \in [\bar{p}_i-1]$, and $j \in [m]$ are arbitrary:
\begin{compactitem}
\item
$\alpha_{i,1}^{z_1}\gamma_{i,1}^{z_2}= \llbracket p_i \rrbracket$,
$\alpha_{i,2}^{z_1}\gamma_{i,2}^{z_2}= \llbracket \bar{p}_i \rrbracket$ and
$\alpha_{i,3}^{z_1}\gamma_{i,3}^{z_2}= \gamma_{i,3}^{z_2} = \llbracket p_i\bar{p}_i \rrbracket^{z_2}$
are fixpoint-free.
\item $\alpha_{i,l,k}^{z_1}\gamma_{i,l,k}^{z_2} = \llbracket p_i\bar{p}_i \rrbracket^{s_{i,l,k}+z_2}$ is fixpoint-free
since $s_{i,l,k}+z_2\equiv l \not \equiv 0 \bmod p_i$ if $\sigma(x_i)=0$ and 
$s_{i,l,k}+z_2\equiv k \not \equiv 0 \bmod \bar{p}_i$ if $\sigma(x_i)=1$.
\item $\beta_j^{z_1}\delta_j^{z_2} = \delta_j^{z_2} = \llbracket r_j \rrbracket^{z_2}$ is fixpoint-free. To see this let
$C_j = \{\tilde{x}_{i_1},\tilde{x}_{i_2},\tilde{x}_{i_3}\}$ and $a \in [3]$ be such that $\sigma(\tilde{x}_{i_a}) = 1$. Then
\eqref{eq-z2} yields $z_2 \equiv \sigma(\tilde{x}_{i_a}) \equiv 1 \bmod \tilde{p}_{i_a}$ and hence
$z_2 \not \equiv 0 \bmod r_j$. \qed
\end{compactitem}
\end{proof}

\begin{corollary}
It is {\sf NP}-complete to check whether 
$\pi_1 \langle \pi_2 \rangle \cap \fpf(n) \neq \emptyset$ holds
for given $\pi_1,\pi_2 \in \Sym(n)$ with $\pi_1 \pi_2 = \pi_2 \pi_1$.
\end{corollary}

\begin{proof}
For $\pi_1$ and $\pi_2$ from the proof of Theorem~\ref{theoremfpfnpcomplete},
there are $z_1, z_2 \in \mathbb{N}$ with $\pi_1^{z_1}\pi_2^{z_2} \in \fpf(n)$ if and only if there is $z \in \mathbb{N}$ with
$\pi_1\pi_2^{z}\in \fpf(n)$.
\qed
\end{proof}


\begin{thebibliography}{10}

\bibitem{AroBar09}
S.~Arora and B.~Barak.
\newblock {\em Computational Complexity - A Modern Approach}.
\newblock Cambridge University Press, 2009.

\bibitem{Arvind13}
V.~Arvind.
\newblock The parameterized complexity of fixpoint free elements and bases in
  permutation groups.
\newblock In {\em Proceedings of the 8th International Symposium on
  Parameterized and Exact Computation, {IPEC} 2013}, volume 8246 of {\em
  Lecture Notes in Computer Science}, pages 4--15. Springer, 2013.

\bibitem{BaLuSe87}
L.~Babai, E.~M. Luks, and {\'{A}}.~Seress.
\newblock Permutation groups in {NC}.
\newblock In {\em Proceedings of the 19th Annual {ACM} Symposium on Theory of
  Computing, STOC 1987}, pages 409--420. {ACM}, 1987.

\bibitem{Bona22}
M.~B{\'{o}}na.
\newblock {\em Combinatorics of Permutations, Third Edition}.
\newblock Discrete mathematics and its applications. {CRC} Press, 2022.

\bibitem{BuchheimJ05}
C.~Buchheim and M.~J{\"{u}}nger.
\newblock Linear optimization over permutation groups.
\newblock {\em Discret. Optim.}, 2(4):308--319, 2005.

\bibitem{Cam10}
P.~J. Cameron.
\newblock {\em Permutation groups}.
\newblock Cambridge University Press, 2010.

\bibitem{cameron2011}
P.~J. Cameron.
\newblock Lectures on derangements.
\newblock In {\em Pretty Structures Conference in Paris}, volume 654, 2011.

\bibitem{CameronW10}
P.~J. Cameron and T.~Wu.
\newblock The complexity of the weight problem for permutation and matrix
  groups.
\newblock {\em Discrete Mathematics}, 310(3):408--416, 2010.

\bibitem{CookM87}
S.~A. Cook and P.~McKenzie.
\newblock Problems complete for deterministic logarithmic space.
\newblock {\em Journal of Algorithms}, 8(3):385--394, 1987.

\bibitem{FurstHL80}
M.~L. Furst, J.~E. Hopcroft, and E.~M. Luks.
\newblock Polynomial-time algorithms for permutation groups.
\newblock In {\em Proceedings of the 21st Annual Symposium on Foundations of
  Computer Science, {FOCS} 1980}, pages 36--41. {IEEE} Computer Society, 1980.

\bibitem{gareyjohnson}
M.~R. Garey and D.~S. Johnson.
\newblock {\em Computers and Intractability: A Guide to the Theory of
  {NP}--completeness}.
\newblock Freeman, 1979.

\bibitem{HesseAB02}
W.~Hesse, E.~Allender, and D.~A.~M. Barrington.
\newblock Uniform constant-depth threshold circuits for division and iterated
  multiplication.
\newblock {\em Journal of Computer and System Sciences}, 65(4):695--716, 2002.

\bibitem{Jordan1872}
C.~Jordan.
\newblock Recherches sur les substitutions.
\newblock {\em Journal de Math\'ematiques Pures et Appliqu\'ees}, 17:351--387,
  1872.

\bibitem{Sims70}
C.~C. Sims.
\newblock Computational methods in the study of permutation groups.
\newblock In {\em Computational Problems in Abstract Algebra}, pages 169--183.
  Pergamon, 1970.

\end{thebibliography}

\end{document}